\documentclass[11pt, reqno]{amsart}
\usepackage{amssymb,mathrsfs,graphicx,extpfeil}
\usepackage{epsfig}
\usepackage{indentfirst, latexsym, amssymb, enumerate,amsmath,graphicx}
\usepackage{float}
\usepackage{colortbl}
\usepackage{epsfig,subfigure}
\usepackage{mathtools}
\usepackage{amsmath}

\title[Synchronization for the inertial Kuramoto model]{On the exponential synchronization for the asymmetric second-order Kuramoto model}

\author[Zhu]{Tingting Zhu \textsuperscript{\MakeLowercase{a,b}}}

\author[Zhang]{Xiongtao Zhang \textsuperscript{\MakeLowercase{c},*}}

\newtheorem{theorem}{Theorem}[section]
\newtheorem{lemma}{Lemma}[section]

\newtheorem{definition}{Definition}[section]

\def\charf {\mbox{{\text 1}\kern-.30em {\text l}}}

\begin{document}

\date{\today}


\thanks{\textsuperscript{a} School of Artificial Intelligence and Big Data, Hefei University, Hefei, 230601, China (ttzhud201880016@163.com)}
\thanks{\textsuperscript{b} Key Laboratory of Applied Mathematics and Artificial Intelligence Mechanism, Hefei University, Hefei, 230601, China}
\thanks{\textsuperscript{c} School of Mathematics and Statistics, Wuhan University, Wuhan, 430072, China (zhangxt@whu.edu.cn)}
\thanks{\textsuperscript{*} Corresponding author. }

\begin{abstract}
In this paper, we study the synchronization problem of nonuniform second-order Kuramoto model with homogeneous dampings and frustration effects on an asymmetric network. More precisely, we focus on the second order model defined on an asymmetric graph with depth no greater than two and present theories on the complete frequency synchronization. Due to the absence of the gradient flow structure, we develop novel energy functions to control the diameters of phase and frequency respectively, which allows us to construct first-order Gronwall-type inequalities. This eventually gives rise to the exponential convergence to the synchronized state in a regime in terms of large coupling strength, small inertia and frustration.
\end{abstract}
\maketitle \centerline{\date}

\section{Introduction}\label{sec:1}
\vspace{0.5cm}

Synchronization is a widespread phenomenon that has been observed across various fields, such as  the circadian rhythms, neural networks, synchronous firing of biological oscillators \cite{B-J-Y84, B-B66, Do-B12, E-K91, M-S90}. Numerous mathematical models have been proposed to illustrate this phenomenon \cite{E91, K75, W67}. The Kuramoto model \cite{K75} has got considerable attention among these models. Theoretical studies on synchronization for the first-order Kuramoto model have been extensively explored, for example, phase-locking and critical coupling \cite{C-H-J-K12,H-K-P15, H-K-R16, H-R20}, network topology \cite{H-L14, Z-Z23}, phase shifts \cite{H-K-L14,  H-K-P18, S-K86}, and time delay \cite{D-H-K20, H-J-K-U20}. Additionally, it is essential to incorporate the effect of inertia into the Kuramoto model in many applications like power systems \cite{C-C-C95, G-W-Y-Z-X-Y20} and Josephson junction array \cite{T-S-S05, W-S97}. In this paper, we focus on the Kuramoto model with inertia effects and phase shifts on an asymmetric network. 

To set up the stage, consider a graph defined by $\mathcal{G} = (V, E, \Psi)$, consisting of a vertex set $V = \{1,2,\ldots, N\}$, an edge set $E \subset V \times V$, and weighted matrix $\Psi= (\psi_{ij})_{N\times N}$ whose element $\psi_{ij}$ denotes the capacity of the communication weight from vertex $j$ to $i$. Note that $\psi_{ij}> 0$ iff $(j,i) \in E$.
The Kuramoto oscillators are assumed to be located at the vertices of graph $G$, and let $\theta_i = \theta_i(t)$ and $\omega_i(t) = \dot{\theta}_i(t) $ be the phase and instantaneous frequency of the $i$-th oscillator. Then, consider the following second-order Kuramoto oscillators with the dynamics given by
\begin{equation}\label{s_phs_com}
m_i \ddot{\theta}_i(t) + d_i\dot{\theta}_i(t) = \Omega_i + \frac{K}{N} \sum_{k=1}^N \psi_{ik} \sin (\theta_k(t) - \theta_i(t) + \alpha_{ik}), \quad t \ge 0,  \ i=1,2,\ldots,N,
\end{equation}
subject to the initial history
\begin{equation*}
\theta_i(0) = \theta_{i0}, \quad \omega_i(0) = \omega_{i0}.
\end{equation*}
Here, $\Omega_i$ indicates the natural frequency of $i$-th oscillator, $m_i > 0$ stands for the strength of inertia, $d_i > 0$ denotes the damping coefficient, and $0 \le \alpha_{ik} < \frac{\pi}{2}$ represents the effect of frustration, and $K > 0$ is the coupling strength. Note that $\alpha_{ii} =0$ for all $1 \le i \le N$ is assumed.
The rich dynamics of model \eqref{s_phs_com} has been widely investigated owing to the potential applications such as in power grids and complex networks  \cite{D-B12, D-C-B13, H-C-Y-S99, L-X-Y14, T-L-O97, Ta-L-O97}.

In recent years, several efforts have been made to address the synchronization analysis of \eqref{s_phs_com}. For the setting $m_i = m, d_i=1$ on an all-to-all coupled network $\psi_{ik} =1$, the model \eqref{s_phs_com} without frustration ($\alpha_{ik} =0$) was analyzed in \cite{C-H-Y11}, providing a frequency synchronization theory for a limited class of initial configurations. The relaxation dynamics of phase-locked state was explored in \cite{C-H-N13}. The authors in \cite{H-J-K19, H-J-K-M22} demonstrated that the initial aggregation of the majority of oscillators can achieve a complete frequency synchronization. For a more general network, sufficient frameworks contributing to synchronized behaviors for the uniform second-order Kuramoto ensemble on a symmetrically connected network were investigated in \cite{C-L-H-X-Y14, W-Q17}. Moreover, the nonuniform second-order phase model with the effects of damping and no frustration were considered  in \cite{C-D-H23, C-H-M18, C-L19, L-X-Y14, W-C-B22}. For instance, the authors in \cite{L-X-Y14} dealt with the case of homogeneous damping $\frac{m_i}{d_i} = \frac{m_j}{d_j}$ for all $i,j$, and sufficient conditions leading to frequency synchronization were presented. For the case of heterogeneous dampings, the authors in \cite{C-L19} showed the convergence of frequency synchronization by exploiting the gradient flow theory following from a symmetric and connected graph topology. Note that most of the above mentioned results are mainly based on the symmetric structure of system and the second-order gradient flow theory, and there is no much information on the convergence rate. Additionally, introducing the frustration term into the sinusoidal function in \eqref{s_phs_com} and non-all-to-all interaction inevitably cause the absence of second-order gradient-like formulation such that the energy fuction method used in \cite{C-L19,L-X-Y14} can not be directly applied. We refer to \cite{D-Z-P-L-J18, F-B-J-K23, Ha-K-L14, L-H16} for the results on the interplay of inertia and frustration. To the best of the authors' knowledge, the theoretical results for model \eqref{s_phs_com} with frustration effects and non-symmetric interaction are relatively few.
 
In this work, we study the emergence of synchronization of the second-order Kuramoto model \eqref{s_phs_com} on a typical asymmetric interaction network:  an asymmetric network with depth not exceeding two (see \eqref{com_net}).  For simplicity, we make the following assumptions
\begin{equation}\label{m/d_equal}
\frac{m_i}{d_i} = \frac{m_j}{d_j} = \gamma, \quad \text{for all } \  i \ne j.
\end{equation}
Based on \eqref{m/d_equal}, the nonuniform second-order Kuramoto model \eqref{s_phs_com} is rewritten in a simpler form as follows
\begin{equation}\label{second_theta_system}
\gamma \ddot{\theta}_i(t) + \dot{\theta}_i(t) = \frac{\Omega_i}{d_i} + \frac{K}{N} \sum_{k=1}^N \frac{\psi_{ik}}{d_i} \sin (\theta_k(t) - \theta_i(t) + \alpha_{ik}).
\end{equation}
When $\gamma=0$, the equation \eqref{second_theta_system} reduces to the well-known Kuramoto-type model, for which numerous works have been done to establish the emergence of asymptotic synchronization when the initial phase diameter is less than $\pi$. When $\gamma$ is small, it is intuitive to view the second-order system \eqref{second_theta_system} as an approximation to the original first order model. This naturally leads to the following question:
\vspace{2mm}
\begin{itemize}
 \item When initial phase diameter is less than $\pi$, dose the synchronization occur in \eqref{second_theta_system}?
\end{itemize}
\vspace{2mm}

This is the motivation of the present paper and we briefly discuss the challenges and outline our main contributions. The first challenge comes from the absence of the symmetry, so that we cannot apply the second order gradient flow method used in  \cite{L-X-Y14} to gain the dissipation. Secondly, as we intend to study the emergence of synchronization in \eqref{second_theta_system} when initial phase diameter is less than $\pi$, it is proper to directly study the dynamic behavior of the phase diameter. However, the dynamic of the phase diameter is not expected to be governed by a second order differential inequality, since the diameter is only Lipschitz continuous. Therefore, it's nontrivial to yield the dissipation of the phase diameter in the second-order system \eqref{second_theta_system}. To deal with these difficulties, we develop novel energy functions (see \eqref{energy_1} and \eqref{energy_2}) involving  the phase diameter, frequency diameter, acceleration diameter and jerk diameter (the diameter of the derivative of the accelerations). With some delicate estimates, we successfully establish the first order Gronwall type inequalities of these energy functions, which eventually yield the exponential emergence of the synchronization for large coupling strength, small frustration effect and over-damping.   
%
%
%
%

The rest of the paper is organized as follows. In Section \ref{sec:2}, we introduce some notations, and state our framework and main result. In Section \ref{sec:3}, we present the emergence of complete synchronization for the nonuniform second-order Kuramoto oscillators with homogeneous dampings and heterogeneous frustrations under the network structure \eqref{com_net}. In Section \ref{sec:5}, we provide some simulations to substantiate our main result. Section \ref{sec:6} is delicated to a brief summary of our work and discussion of the future issue.

\section{Preliminaries}\label{sec:2}
\setcounter{equation}{0}
In this section, we present several notations which will be frequently used throughout the paper, and provide our main result on the frequency synchronization for the second-order Kuramoto model \eqref{second_theta_system}. 

\subsection{Notations}
For notational simplicity, we use $\theta(t)$, $\omega(t)$, $a(t)$ and  $b(t)$ to represent the phase, frequency, acceleration, and jerk vectors respectively. To be specific, we have
\begin{equation*}
\begin{aligned}
&\theta(t) = (\theta_1(t), \theta_2(t), \ldots,\theta_N(t)), \quad \omega(t) = (\omega_1(t), \omega_2(t), \ldots, \omega_N(t)), \\
& a(t) = (a_1(t), a_2(t),\ldots,a_N(t)), \quad b(t) = (b_1(t), b_2(t), \ldots, b_N(t)),
\end{aligned}
\end{equation*}
where the following relations hold naturally
\begin{equation}\label{defi_wab}
\omega_i(t) := \dot{\theta}_i(t), \quad a_i(t) := \dot{\omega}_i(t), \quad b_i(t) := \dot{a}_i(t).
\end{equation}
Then, the corresponding parameters such as diameters of phase, frequency, acceleration and jerk are defined as follows,
\begin{equation*}
\begin{aligned}
&D_\theta(t) = \max_{1 \le i \le N} \theta_i(t) - \min_{1 \le i \le N} \theta_i(t), \quad D_\omega(t) = \max_{1 \le i \le N} \omega_i(t) - \min_{1 \le i \le N} \omega_i(t),\\
&D_a(t) = \max_{1 \le i \le N} a_i(t) - \min_{1 \le i \le N} a_i(t), \quad D_b(t) = \max_{1 \le i \le N} b_i(t) - \min_{1 \le i \le N} b_i(t),\\
&D_\Omega = \max_{1 \le i \le N} \frac{\Omega_i}{d_i} - \min_{1 \le i \le N} \frac{\Omega_i}{d_i}, \quad \bar{\alpha} = \max_{1\le i,j \le N} \alpha_{ij}. 
\end{aligned}
\end{equation*}
We recall the concept of the complete synchronization for the Kuramoto model.

\begin{definition}
Let $\theta(t) = (\theta_1(t), \ldots, \theta_N(t))$ be a phase vector of Kuramoto oscillators. The Kuramoto ensemble exhibits complete frequency synchronization asymptotically iff the relative frequency difference tends to zero asymptotically:
\begin{equation*}
\lim_{t \to + \infty} |\omega_i(t) - \omega_j(t)| = 0, \quad \forall \ i \ne j.
\end{equation*}
\end{definition}

\subsection{Main result} 
In this part, we deliberate our main result on the asymptotic synchronization for the model \eqref{second_theta_system}. We introduce the following network structure:
\begin{equation}\label{com_net}
\mathcal{C}:= \min_{1 \le i\ne j \le N} \left\{ \frac{\psi_{ij}}{d_i} + \frac{\psi_{ji}}{d_j} +\sum_{\substack{k = 1 \\ k \ne i,j}}^N \min \left\{\frac{\psi_{ik}}{d_i}, \frac{\psi_{jk}}{d_j}\right\} \right\} > 0,
\end{equation}
which means that the underlying graph can be asymmetric but its depth is less than or equal to two.
Based on the network topology \eqref{com_net}, we set
\begin{equation*}
\psi_u = \max_{1 \le i \neq j \le N} \frac{\psi_{ij}}{d_i},
\end{equation*}
and introduce the following Assumption $(\mathcal{A})$:

$\bullet$ $(\mathcal{A}_1)$: There exists a positive constant $\beta \in (0,\pi)$ such that the following condition holds for initial configuration
\begin{equation}\label{initial_con}
2\gamma (D_\Omega+ 2K \psi_u \sin \bar{\alpha}) + (1 + 4\gamma K \psi_u ) D_\theta(0) + 3\gamma D_\omega(0) < \beta < \pi.
\end{equation}
This condition shows that if the phase diameter is initially close to $\pi$, the inertia or the initial frequency diameter is required to be small. In the inertial model, this is a natural assumption if we want to keep all the oscillators staying in the half circle.

 $\bullet$ $(\mathcal{A}_2)$: Based on Assumption  $(\mathcal{A}_1)$, we further assume that there exists a positive constant $D^\infty < \min \left\{\beta, \frac{\pi}{2} \right\}$ such that the following conditions hold for the parameters defined previously
\begin{align}
&D^\infty + \bar{\alpha} < \frac{\pi}{2},\label{frustration_con}\\
&\gamma K < \min\left\{\frac{\mathcal{C}\cos \bar{\alpha} \sin \beta}{32N\psi_u^2 \beta}, \frac{N \beta}{\mathcal{C} \cos \bar{\alpha} \sin \beta},  \frac{\mathcal{C}\cos (D^\infty + \bar{\alpha})}{32N \psi^2_u}, \frac{2N }{\mathcal{C}\cos (D^\infty + \bar{\alpha})}\right\} ,   \label{gammaK_con}\\
& K > \max\left\{ \frac{8N\beta\left(D_\Omega + 2K \psi_u\sin\bar{\alpha} \right)}{D^\infty \mathcal{C}  \cos \bar{\alpha} \sin \beta}, \frac{8N\mu}{\mathcal{C} \cos (D^\infty + \bar{\alpha})}\right\}, \label{K_con}
\end{align}
where $\mu$ is given by 
\begin{equation}\label{mu}
\mu := \frac{64N^2\psi^2_u\beta^2(D_\Omega +  2K \psi_u \sin \bar{\alpha}) (D^\infty + \sin \bar{\alpha})}{\mathcal{C}^2  \cos^2 \bar{\alpha} \sin^2 \beta}.
\end{equation}
Obviously, these conditions can be ensured for small $\gamma$, $\bar{\alpha}$ and large $K$.
Under the network structure \eqref{com_net}, our main result in this paper is presented as below.
\begin{theorem}\label{com_main}
Let $\theta(t)$ be a solution to system \eqref{second_theta_system} with the network structure \eqref{com_net}, and suppose the Assumption $(\mathcal{A})$ holds. Then,  we have 
\begin{equation*}
D_\omega(t)  \le C e^{-\Lambda t}, \quad t \ge 0,
\end{equation*}
where $C, \Lambda$ are positive constants depending on initial data and system parameters.
\end{theorem}

\section{Proof of the Main Result}\label{sec:3}
\setcounter{equation}{0}

In this section, we will address the synchronization problem of system \eqref{second_theta_system} under the network interaction \eqref{com_net}, and provide the detailed proof of Theorem \ref{com_main}. Our approach relies on two novel energy functions, which govern the dynamics of the phase diameter and frequency diameter, respectively. 

\subsection{Phase cohesiveness}\label{sec:3.1} 
In this part, we will show that all oscillators will concentrate into an arc of a quarter circle in finite time.
To set up the stage, we define a crucial energy function
\begin{equation}\label{energy_1}
\begin{aligned}
E_1(t) = D_\theta(t) + \frac{\mathcal{C}\cos \bar{\alpha}\sin \beta}{4N\psi_u \beta }\gamma D_\omega(t) + 2\gamma^2 D_a(t),
\end{aligned}
\end{equation}
Moreover, due to the analyticity of system \eqref{s_phs_com}, we can divide the whole time interval into countably many intervals such as
\begin{equation}\label{com_time_divide}
[0,+\infty) = \bigcup_{l=1}^{+\infty} J_l, \quad J_l = [t_{l-1},t_l), \quad t_0 = 0,
\end{equation}
such that in each time interval $J_l$, the orders of oscillators' phases, frequencies and accelerations are unchanged. The following lemma shows an a priori estimate on the initial energy.

\begin{lemma}\label{com_ini_esti}
Let $(\theta(t), \omega(t))$ be a solution to system \eqref{second_theta_system} with the network structure \eqref{com_net}, and suppose the initial data satisfy the condition \eqref{initial_con}.
Then, we have
\begin{equation*}
E_1(0) <\beta  < \pi.
\end{equation*}
\end{lemma}
\begin{proof}
Note that $\omega_i(t) = \dot{\theta}_i(t)$ and $a_i(t) = \dot{\omega}_i(t)$, then we see from \eqref{second_theta_system} that
\begin{equation*}
\gamma a_i(0) + \omega_i(0) = \frac{\Omega_i}{d_i}  + \frac{K}{N} \sum_{k=1}^N \frac{\psi_{ik}}{d_i} \sin (\theta_k(0) - \theta_i(0) + \alpha_{ik}),
\end{equation*}
which implies 
\begin{equation}\label{a_initial}
a_i(0) = \frac{1}{\gamma} \left[- \omega_i(0) + \frac{\Omega_i}{d_i}  + \frac{K}{N} \sum_{k=1}^N \frac{\psi_{ik}}{d_i} \sin (\theta_k(0) - \theta_i(0) + \alpha_{ik}) \right].
\end{equation}
It is easy to see that we can find some indices $i,j$ such that
\begin{equation*}
D_a(0) = a_i(0) - a_j(0).
\end{equation*}
Then, direct calculations show that
\begin{equation*}
\begin{aligned}
D_a(0) &= -\frac{1}{\gamma} (\omega_i(0) - \omega_j(0)) + \frac{1}{\gamma} \left(\frac{\Omega_i}{d_i} - \frac{\Omega_j}{d_j}\right) \\
&\quad+ \frac{K}{\gamma N} \sum_{k=1}^N \frac{\psi_{ik}}{d_i} \sin (\theta_k(0) - \theta_i(0) + \alpha_{ik}) - \frac{K}{\gamma N} \sum_{k=1}^N \frac{\psi_{jk}}{d_j} \sin (\theta_k(0) - \theta_j(0) + \alpha_{jk})\\
&\le \frac{1}{\gamma} D_\omega(0) + \frac{1}{\gamma}D_\Omega + \frac{2K\psi_u }{\gamma} D_\theta(0) + \frac{2K \psi_u \sin \bar{\alpha}}{\gamma}.
\end{aligned}
\end{equation*}
Here, we used the following relations
\begin{equation*}
\sin (x+y) = \sin x \cos y + \cos x \sin y, \quad |\sin x| \le |x|, \quad \cos \alpha_{ik} \le 1, \quad \sin \alpha_{ik} \le \sin \bar{\alpha}.
\end{equation*}
Therefore, together with the above estimate on $D_a(0)$, we derive
\begin{equation*}
\begin{aligned}
&E_1(0) = D_\theta(0) + \frac{\mathcal{C}\cos \bar{\alpha}\sin \beta}{4N\psi_u\beta }\gamma D_\omega(0) + 2\gamma^2 D_a(0)\\
&\le D_\theta(0) + \gamma D_\omega(0) + 2\gamma^2 \left(\frac{1}{\gamma} D_\omega(0) + \frac{1}{\gamma}D_\Omega + \frac{2K\psi_u }{\gamma} D_\theta(0) + \frac{2K \psi_u \sin \bar{\alpha}}{\gamma} \right)\\
&= 2\gamma (D_\Omega+ 2K \psi_u \sin \bar{\alpha}) + (1 + 4\gamma K \psi_u ) D_\theta(0) + 3\gamma D_\omega(0) < \beta  < \pi.
\end{aligned}
\end{equation*}
\end{proof}

Next, according to \eqref{com_time_divide}, we may provide a second-order differential inequality of phase diameter function $D_\theta(t)$  on each time interval $J_l$. Then we have the following lemma.
\begin{lemma}\label{second_theta_eq}
Let $\theta(t)$ be a solution to system \eqref{second_theta_system} with the network structure \eqref{com_net}, and suppose
\begin{equation*}
D_\theta(t) < \beta < \pi, \quad \text{for} \ t \in J_l,
\end{equation*}
where $J_l$ is some time interval defined in \eqref{com_time_divide}.
Then, we have
\begin{equation}\label{phsdia_s_dynamic}
\begin{aligned}
\gamma \ddot{D}_\theta(t) + \dot{D}_\theta(t)  \le D_\Omega + 2K \psi_u\sin\bar{\alpha} - \frac{K\mathcal{C}  \cos \bar{\alpha} \sin \beta}{N\beta} D_\theta(t), \quad t \in J_l.
\end{aligned} 
\end{equation}
\end{lemma}
\begin{proof}
For $t \in J_l$, we can find some indices $i$ and $j$ such that
\begin{equation*}
D_\theta(t) = \theta_i(t) - \theta_j(t).
\end{equation*}
It follows from \eqref{second_theta_system} that
\begin{equation*}
\begin{aligned}
&\gamma(\ddot{\theta}_i(t) - \ddot{\theta}_j(t)) + (\dot{\theta}_i(t) - \dot{\theta}_j(t)) \\
&= \frac{\Omega_i}{d_i} - \frac{\Omega_j}{d_j} + \frac{K}{N} \sum_{k=1}^N \frac{\psi_{ik}}{d_i}[\sin(\theta_k - \theta_i) \cos \alpha_{ik} + \cos(\theta_k - \theta_i) \sin \alpha_{ik}]\\
&\quad - \frac{K}{N} \sum_{k=1}^N \frac{\psi_{jk}}{d_j} [\sin (\theta_k - \theta_j) \cos \alpha_{jk} + \cos (\theta_k - \theta_j) \sin \alpha_{jk}]\\
&\le D_\Omega + 2K\psi_u \sin \bar{\alpha} + \frac{K\cos \bar{\alpha}}{N} \sum_{k=1}^N \left(\frac{\psi_{ik}}{d_i} \sin (\theta_k- \theta_i) -  \frac{\psi_{jk}}{d_j} \sin (\theta_k - \theta_j)\right).
\end{aligned}
\end{equation*}
where we utilized the relations
\begin{equation*}
 0 \le \sin \alpha_{ik} \le \sin \bar{\alpha}, \quad \cos \alpha_{ik} > \cos \bar{\alpha} > 0.
\end{equation*}
Furthermore, we deal with the last term in above formula
\begin{equation*}
\begin{aligned}
&\sum_{k=1}^N \left(\frac{\psi_{ik}}{d_i} \sin (\theta_k- \theta_i) -  \frac{\psi_{jk}}{d_j} \sin (\theta_k - \theta_j)\right)\\
&= - \left(\frac{\psi_{ij}}{d_i} + \frac{\psi_{ji}}{d_j}\right) \sin (\theta_i - \theta_j) - \sum_{\substack{k = 1 \\ k \ne i,j}}^N \min \left\{\frac{\psi_{ik}}{d_i}, \frac{\psi_{jk}}{d_j}\right\} [\sin (\theta_i - \theta_k) + \sin (\theta_k - \theta_j)]\\
&\le - \mathcal{C} \sin D_\theta(t),
\end{aligned}
\end{equation*}
where $\mathcal{C}$ is defined in \eqref{com_net}. Then, we collect all above estimates to derive
\begin{equation*}
\begin{aligned}
\gamma \ddot{D}_\theta(t) + \dot{D}_\theta(t) &\le D_\Omega + 2K\psi_u \sin \bar{\alpha} - \frac{K\mathcal{C}  \cos \bar{\alpha}}{N}\sin D_\theta(t)\\
&\le D_\Omega + 2K \psi_u\sin\bar{\alpha} - \frac{K\mathcal{C}  \cos \bar{\alpha} \sin \beta}{N\beta} D_\theta(t), \quad t \in J_l.
\end{aligned}
\end{equation*}
where we leveraged the monotone decreasing property of $\frac{\sin x}{x}$ in $(0, \pi]$. 
\end{proof}

It is worth mentioning that the estimate on the second-order dynamics of $D_\theta(t)$ in Lemma \ref{second_theta_eq} does not suffice to show the phase cohesiveness. The reason arises from the singularities of second-order derivative of $D_\theta(t)$ in view of the whole time line. Hence, we next provide an estimate on the dynamics of $D_a(t)$, which can control the term $\ddot{D}_\theta(t)$ in some degree.

To this end, recall $\omega_i(t) = \dot{\theta}_i(t)$ and it yields from \eqref{second_theta_system} that
\begin{equation}\label{first_fre}
\gamma \dot{\omega}_i(t) + \omega_i(t) = \frac{\Omega_i}{d_i} + \frac{K}{N} \sum_{j=1}^N \frac{\psi_{ij}}{d_i}\sin (\theta_j(t) - \theta_i(t)+ \alpha_{ij}), \quad t > 0,  \ i=1,2,\ldots,N.
\end{equation}
Directly differentiating system \eqref{first_fre}, we obtain the second-order frequency model
\begin{equation}\label{inertia_fre}
\gamma \ddot{\omega}_i(t) + \dot{\omega}_i(t) = \frac{K}{N} \sum_{j=1}^N \frac{\psi_{ij}}{d_i} \cos (\theta_j(t) - \theta_i(t) + \alpha_{ij})(\omega_j(t) - \omega_i(t)), \quad t \ge 0,  \ i=1,2,\ldots,N.
\end{equation}
Moreover, according to $a_i(t) = \dot{\omega}_i(t)$, we have
\begin{equation}\label{first_acce}
\gamma \dot{a}_i(t) + a_i(t) =  \frac{K}{N} \sum_{j=1}^N \frac{\psi_{ij}}{d_i} \cos (\theta_j(t) - \theta_i(t) + \alpha_{ij})(\omega_j(t) - \omega_i(t)), \quad t \ge 0,  \ i=1,2,\ldots,N.
\end{equation}

\begin{lemma}\label{first_a_eq}
Let $(\theta(t),\omega(t))$ be a solution to system \eqref{second_theta_system} with the network structure \eqref{com_net}. Then, we have
\begin{equation}\label{accedia_f_dynamics}
\begin{aligned}
\gamma\dot{D}_a(t) + D_a(t) \le 2K \psi_uD_\omega(t), \quad t \in J_i, \ i=1,2,\ldots,
\end{aligned}
\end{equation}
where $J_i$ is defined in \eqref{com_time_divide}.
\end{lemma}
\begin{proof}
For any fixed time interval $J_i$ and any $t \in J_i$ with $i=1,2,\ldots$,, we can find some indices $k$ and $l$ such that
\begin{equation}\label{E-2}
D_a(t) = a_k(t) - a_l(t).
\end{equation}
It is easy to see from \eqref{first_acce} that
\begin{equation*}
\begin{aligned}
&\gamma(\dot{a}_k(t) - \dot{a}_l(t)) + (a_k(t) - a_l(t)) \\
&= \frac{K}{N} \sum_{j=1}^N \frac{\psi_{kj}}{d_k} \cos (\theta_j - \theta_k + \alpha_{kj})(\omega_j - \omega_k)- \frac{K}{N} \sum_{j=1}^N \frac{\psi_{lj}}{d_l} \cos (\theta_j - \theta_l + \alpha_{lj})(\omega_j - \omega_l).
\end{aligned}
\end{equation*}
Then from \eqref{E-2}, we use $|\cos x| \le 1$ to get
\begin{equation*}
\begin{aligned}
\gamma\dot{D}_a(t) + D_a(t) \le \frac{K}{N} \sum_{j=1}^N \psi_u|\omega_j - \omega_k| + \frac{K}{N}\sum_{j=1}^N \psi_u |\omega_j - \omega_l| \le 2K \psi_uD_\omega(t), \quad t \in J_i.
\end{aligned}
\end{equation*}
Owing to the arbitrary choice of $J_i$, we complete the proof.
\end{proof}

To control the term $\ddot{D}_\theta(t)$, we deduce the estimate on the dynamics of $D_a(t)$. However, there is one more non-dissipative term $D_\omega(t)$ displayed in Lemma \ref{first_a_eq}. Thus, we subsequently analyze the dynamics of $D_\omega(t)$.

\begin{lemma}\label{first_w_differential}
Let $(\theta(t),\omega(t))$ be a solution to system \eqref{second_theta_system} with the network structure \eqref{com_net}. Then, we have 
\begin{equation}\label{fredia_f_dynamics}
\begin{aligned}
\gamma\dot{D}_\omega(t) + D_\omega(t) &\le D_\Omega +  2K \psi_u \sin \bar{\alpha} + 2K \psi_u  D_\theta(t), \quad t \in J_l, \ l =1,2,\ldots,
\end{aligned}
\end{equation}
where $J_l$ is defined in \eqref{com_time_divide}.
\end{lemma}
\begin{proof}
For any fixed time interval $J_l$ and any $t \in J_l$ with $l =1,2,\ldots$, we can find some indices $p$ and $q$ such that
\begin{equation}\label{E-1}
D_\omega(t) = \omega_p(t) - \omega_q(t).
\end{equation}
We see from \eqref{first_fre} that
\begin{equation*}
\begin{aligned}
&\gamma (\dot{\omega}_p(t) - \dot{\omega}_q(t)) + \omega_p(t) - \omega_q(t) \\
&= \frac{\Omega_p}{d_p} - \frac{\Omega_q}{d_q} + \frac{K}{N} \sum_{k=1}^N \frac{\psi_{pk}}{d_p}[\sin (\theta_k - \theta_p) \cos \alpha_{pk} +  \cos (\theta_k - \theta_p) \sin \alpha_{pk}]\\
&\quad - \frac{K}{N} \sum_{k=1}^N \frac{\psi_{qk}}{d_q} [\sin (\theta_k - \theta_q) \cos \alpha_{qk} + \cos (\theta_k - \theta_q) \sin \alpha_{qk}].\\
\end{aligned}
\end{equation*}
Then from \eqref{E-1}, we apply $|\sin x| \le |x|$ to obtain
\begin{equation*}
\begin{aligned}
&\gamma\dot{D}_\omega(t) + D_\omega(t) \\
&\le D_\Omega + 2K \psi_u \sin \bar{\alpha} +  \frac{K}{N}\sum_{k=1}^N \frac{\psi_{pk}}{d_p} |\sin (\theta_k - \theta_p)| + \frac{K }{N} \sum_{k=1}^N \frac{\psi_{qk}}{d_q} |\sin (\theta_k - \theta_q)| \\
&\le D_\Omega +  2K \psi_u \sin \bar{\alpha} + 2K \psi_u  D_\theta(t), \quad t \in J_l. 
\end{aligned}
\end{equation*}
Due to the arbitrary choice of $J_l$, the proof is completed.
\end{proof}

Now, we are ready to present a Gronwall type inequality on the energy function $E_1(t)$.

\begin{lemma}\label{Lm3-5}
Let $(\theta(t),\omega(t))$ be a solution to system \eqref{second_theta_system} with the network structure \eqref{com_net}, and suppose the assumptions \eqref{initial_con}, \eqref{gammaK_con} and \eqref{K_con} are fulfilled. Then, we have
\begin{equation}\label{phs_functional_dynamics}
\begin{aligned}
\frac{d}{dt} E_1(t) \le 2 \left(D_\Omega + 2K \psi_u\sin\bar{\alpha} \right)  - \frac{K\mathcal{C}  \cos \bar{\alpha} \sin \beta}{2N\beta} E_1(t), \quad \text{a.e.}\ t \ge 0.
\end{aligned}
\end{equation}
\end{lemma}
\begin{proof}
We prove this by continuity argument and the proof is splited into the following three steps.

\noindent $\bullet$ {\bf Step 1:} First, we construct a set
\begin{equation*}
\mathcal{M} = \left\{ T> 0 \ | \ E_1(t) < \beta < \pi, \ \forall \ 0 \le t < T\right\},
\end{equation*}
where $E_1(t)$ is defined in \eqref{energy_1}.
According to Lemma \ref{com_ini_esti} and continuity of $E_1(t)$, there exists a small constant $\varepsilon_1>0$ such that
\begin{equation*}
E_1(t) <\beta  < \pi, \quad \forall \ t \in [0,\varepsilon_1).
\end{equation*}
This means $\varepsilon_1 \in \mathcal{M}$ and thus the set $\mathcal{M}$ is not empty. We define $T^* = \sup \mathcal{M}$. In the following, we will show $T^* = +\infty$. Suppose not, i.e., $T^* < +\infty$, then we have
\begin{equation}\label{EE-4}
\begin{aligned}
&E_1(t) < \beta < \pi, \ \text{for} \ t \in [0,T^*), \quad \text{and} \quad E_1(T^*) = \beta,\\
\end{aligned}
\end{equation}
which also yields
\begin{equation}\label{EE-5}
D_\theta(t) < \beta < \pi, \quad \text{for} \ 0 \le t < T^*.
\end{equation}
For the sake of discussion, we split the time interval $[0,T^*)$ into finitely many intervals as below,
\begin{equation*}
[0,T^*) = \bigcup_{l=1}^n I_l, \quad I_l = [\tau_{l-1},\tau_l), \quad \tau_0 = 0,
\end{equation*}
so that the orders of oscillators' phases, frequencies and accelerations are fixed on each time interval $I_l$, respectively. On each interval $I_l$, we can apply the same argument as in \eqref{phsdia_s_dynamic}, \eqref{fredia_f_dynamics} and \eqref{accedia_f_dynamics} to obtain that for $t \in I_l$,
\begin{align}
&\gamma \ddot{D}_\theta(t) + \dot{D}_\theta(t)  \le D_\Omega + 2K \psi_u\sin\bar{\alpha} - \frac{K\mathcal{C}  \cos \bar{\alpha} \sin \beta}{N\beta} D_\theta(t),\label{EE-1}\\
&\gamma\dot{D}_\omega(t) + D_\omega(t) \le D_\Omega +  2K \psi_u \sin \bar{\alpha} + 2K \psi_u  D_\theta(t), \label{EE-2}\\
&\gamma\dot{D}_a(t) + D_a(t) \le 2K \psi_uD_\omega(t).\label{EE-3}
\end{align}
Then, we collect \eqref{EE-1}, \eqref{EE-2} and \eqref{EE-3} to find
\begin{equation*}
\begin{aligned}
&\gamma \ddot{D}_\theta(t) + \dot{D}_\theta(t)  + \frac{\mathcal{C}\cos \bar{\alpha}\sin \beta}{4N\psi_u \beta }\gamma\dot{D}_\omega(t) + \frac{\mathcal{C}\cos \bar{\alpha}\sin \beta}{4N\psi_u \beta }D_\omega(t)+ 2\gamma^2\dot{D}_a(t) + 2\gamma D_a(t)\\
&\le D_\Omega + 2K \psi_u\sin\bar{\alpha} - \frac{K\mathcal{C}  \cos \bar{\alpha} \sin \beta}{N\beta} D_\theta(t)\\
&\quad + \frac{\mathcal{C}\cos \bar{\alpha}\sin \beta}{4N\psi_u \beta }(D_\Omega +  2K \psi_u \sin \bar{\alpha})+ \frac{K\mathcal{C}  \cos \bar{\alpha} \sin \beta}{2N\beta} D_\theta(t) + 4\gamma K \psi_uD_\omega(t), \quad t \in I_l.
\end{aligned}
\end{equation*}
This yields
\begin{equation}\label{E-5}
\begin{aligned}
&\frac{d}{dt} \left( D_\theta(t) + \frac{\mathcal{C}\cos \bar{\alpha}\sin \beta}{4N\psi_u \beta }\gamma D_\omega(t) + 2\gamma^2 D_a(t)\right)\\
&\le - \gamma \ddot{D}_\theta(t) - \frac{\mathcal{C} \cos \bar{\alpha}\sin \beta}{4N\psi_u\beta }D_\omega(t) - 2\gamma D_a(t) + D_\Omega + 2K \psi_u\sin\bar{\alpha} - \frac{K\mathcal{C}  \cos \bar{\alpha} \sin \beta}{N\beta} D_\theta(t)\\
&\quad + \frac{\mathcal{C}\cos \bar{\alpha}\sin \beta}{4N\psi_u \beta }(D_\Omega +  2K \psi_u \sin \bar{\alpha})+ \frac{K\mathcal{C}  \cos \bar{\alpha} \sin \beta}{2N\beta} D_\theta(t) + 4\gamma K \psi_uD_\omega(t)\\
&\le 2 \left(D_\Omega + 2K \psi_u\sin\bar{\alpha} \right) -  \frac{K\mathcal{C}  \cos \bar{\alpha} \sin \beta}{2N\beta} D_\theta(t)- \left(\frac{\mathcal{C}\cos \bar{\alpha}\sin \beta}{4N\psi_u\beta } -  4\gamma K \psi_u\right)D_\omega(t) - \gamma D_a(t)\\
&\le 2 \left(D_\Omega + 2K \psi_u\sin\bar{\alpha} \right) -  \frac{K\mathcal{C}  \cos \bar{\alpha} \sin \beta}{2N\beta} D_\theta(t)- \frac{\mathcal{C}\cos \bar{\alpha}\sin \beta}{8N\psi_u\beta }  D_\omega(t) - \gamma D_a(t), \quad t \in I_l,
\end{aligned}
\end{equation}
where we used \eqref{gammaK_con} and the following relations
\begin{equation*}
\begin{aligned}
&|\ddot{D}_\theta(t)| \le  D_a(t), \quad \frac{\mathcal{C}\cos \bar{\alpha}\sin \beta}{4N\psi_u\beta } < 1,\\
& \gamma K < \frac{\mathcal{C} \cos \bar{\alpha}\sin \beta}{32N\psi_u^2 \beta} \quad \Longrightarrow \quad \frac{\mathcal{C} \cos \bar{\alpha}\sin \beta}{8N\psi_u\beta } -  4\gamma K \psi_u > 0.
\end{aligned}
\end{equation*}
We further transform \eqref{E-5} into the following 
\begin{equation*}
\begin{aligned}
&\frac{d}{dt} \left( D_\theta(t) + \frac{\mathcal{C} \cos \bar{\alpha}\sin \beta}{4N\psi_u\beta }\gamma D_\omega(t) + 2\gamma^2 D_a(t) \right)\\
&\le 2 \left(D_\Omega + 2K \psi_u\sin\bar{\alpha} \right)  - \frac{K\mathcal{C}  \cos \bar{\alpha} \sin \beta}{2N\beta} \left(D_\theta(t) + \frac{1}{4K\psi_u }D_\omega(t) + \frac{2N\beta}{K\mathcal{C}  \cos \bar{\alpha} \sin \beta} \gamma D_a(t)\right)  \\
&\le 2 \left(D_\Omega + 2K \psi_u\sin\bar{\alpha} \right)  - \frac{K\mathcal{C}  \cos \bar{\alpha} \sin \beta}{2N\beta}\left( D_\theta(t) + \frac{\mathcal{C} \cos \bar{\alpha}\sin \beta}{4N\psi_u\beta }\gamma D_\omega(t) + 2\gamma^2 D_a(t) \right) , \quad t \in I_l.
\end{aligned}
\end{equation*}
Here, we exploited the assumption \eqref{gammaK_con} yielding
\begin{equation*}
\begin{aligned}
&\gamma K < \frac{N\beta}{\mathcal{C} \cos \bar{\alpha} \sin \beta} \ \Longrightarrow \ \frac{1}{4K\psi_u} > \frac{\mathcal{C}\cos \bar{\alpha}\sin \beta}{4N\psi_u\beta }\gamma \quad \text{and} \quad \frac{2N\beta}{K\mathcal{C}  \cos \bar{\alpha} \sin \beta} > 2 \gamma. 
\end{aligned}
\end{equation*}
Thus, we conclude from the above estimate and \eqref{energy_1} that
\begin{equation}\label{E-5-1}
\begin{aligned}
\frac{d}{dt} E_1(t) \le 2 \left(D_\Omega + 2K \psi_u\sin\bar{\alpha} \right)  - \frac{K\mathcal{C}  \cos \bar{\alpha} \sin \beta}{2N\beta} E_1(t), \quad \text{a.e.} \ t \in [0,T^*).
\end{aligned}
\end{equation}

\noindent $\bullet$ {\bf Step 2:} Next, we study the upper bound of $E_1(t)$ on $[0,T^*)$. From \eqref{E-5-1}, we see that
\begin{equation*}
\begin{aligned}
\frac{d}{dt} E_1(t) 
&\le - \frac{K\mathcal{C}  \cos \bar{\alpha} \sin \beta}{2N\beta}   \left(E_1(t) - \frac{4N\beta\left(D_\Omega + 2K \psi_u\sin\bar{\alpha} \right)}{K\mathcal{C}  \cos \bar{\alpha} \sin \beta} \right), \quad \text{a.e.} \ t \in [0,T^*).
\end{aligned}
\end{equation*}
This implies that
\begin{equation}\label{E-5-2}
E_1(t) \le \max \left\{ E_1(0), \frac{8N\beta\left(D_\Omega + 2K \psi_u\sin\bar{\alpha} \right)}{K\mathcal{C}  \cos \bar{\alpha} \sin \beta}\right\}, \quad \text{for} \ t \in [0,T^*).
\end{equation}
According to Lemma \ref{com_ini_esti}, we have
\begin{equation*}
E_1(0) <  \beta < \pi.
\end{equation*}
Together with the assumption \eqref{K_con} on the coupling strength $K$ that
\begin{equation*}
K > \frac{8N \left( D_\Omega + 2K \psi_u \sin \bar{\alpha}\right)}{\mathcal{C} \cos \bar{\alpha} \sin \beta}\quad \Longrightarrow \quad \frac{8N\beta\left(D_\Omega + 2K \psi_u\sin\bar{\alpha} \right)}{K\mathcal{C}  \cos \bar{\alpha} \sin \beta} < \beta,
\end{equation*}
we get from \eqref{E-5-2} that
\begin{equation}\label{E-5-3}
E_1(t) \le \max \left\{ E_1(0), \frac{8N\beta\left(D_\Omega + 2K \psi_u\sin\bar{\alpha} \right)}{K\mathcal{C}  \cos \bar{\alpha} \sin \beta}\right\} < \beta, \quad \text{for} \ t \in [0,T^*).
\end{equation}

\noindent $\bullet$ {\bf Step 3:}
Due to the Lipschitz continuity of $E_1(t)$, it yields from \eqref{E-5-3} that
\begin{equation*}
E_1(T^*) \le \max \left\{ E_1(0), \frac{8N\beta\left(D_\Omega + 2K \psi_u\sin\bar{\alpha} \right)}{K\mathcal{C}  \cos \bar{\alpha} \sin \beta}\right\} < \beta.
\end{equation*}
This obviously contradicts to $\eqref{EE-4}_2$. Thus, we obtain $T^* = +\infty$, which means
\begin{equation*}
E_1(t) < \beta < \pi,\quad \forall \ 0 \le t < +\infty.
\end{equation*}
Then, we can apply the similar argument in \eqref{E-5-1} to conclude that
\begin{equation*}
\frac{d}{dt} E_1(t) \le 2 \left(D_\Omega + 2K \psi_u\sin\bar{\alpha} \right)  - \frac{K\mathcal{C}  \cos \bar{\alpha} \sin \beta}{2N\beta} E_1(t), \quad \text{a.e.}\ t \ge 0.
\end{equation*}
\end{proof}

Now, we are ready to show that all Kuramoto oscillators will be trapped into a quarter circle in finite time.

\begin{lemma}\label{phs_small_bound}
Let $(\theta(t),\omega(t))$ be a solution to system \eqref{second_theta_system} with the network structure \eqref{com_net}, and suppose the Assumption $(\mathcal{A})$ holds. Then, there exists time $t_* \ge 0$ such that
\begin{equation*}
D_\theta(t) <D^\infty < \frac{\pi}{2}, \quad t\ge t^*.
\end{equation*}
Moreover, we have
\begin{equation*}
D_\omega(t) \le \frac{32N^2\psi_u\beta^2(D_\Omega +  2K \psi_u \sin \bar{\alpha})}{\gamma K\mathcal{C}^2  \cos^2 \bar{\alpha} \sin^2 \beta}, \quad t \ge t_*.
\end{equation*}
\end{lemma}
\begin{proof}
We see from \eqref{phs_functional_dynamics} that 
\begin{equation}\label{E-6}
\begin{aligned}
\frac{d}{dt} E_1(t) &\le 2 \left(D_\Omega + 2K \psi_u\sin\bar{\alpha} \right)  - \frac{K\mathcal{C}  \cos \bar{\alpha} \sin \beta}{2N\beta} E_1(t)\\
&\le - \frac{K\mathcal{C}  \cos \bar{\alpha} \sin \beta}{2N\beta}   \left(E_1(t) - \frac{4N\beta\left(D_\Omega + 2K \psi_u\sin\bar{\alpha} \right)}{K\mathcal{C}  \cos \bar{\alpha} \sin \beta} \right), \quad \text{a.e.} \ t \ge 0.
\end{aligned}
\end{equation}
In the sequel, we consider two different cases, respectively.

\noindent $\diamond$ For one case that
\begin{equation*}
E_1(0) > \frac{8N\beta\left(D_\Omega + 2K \psi_u\sin\bar{\alpha} \right)}{K\mathcal{C}  \cos \bar{\alpha} \sin \beta},
\end{equation*}
we find from \eqref{E-6} that when $E_1(t) \in [\frac{8N\beta\left(D_\Omega + 2K \psi_u\sin\bar{\alpha} \right)}{K\mathcal{C}  \cos \bar{\alpha} \sin \beta}, E_1(0)]$, 
it will be strictly decreasing as below
\begin{equation*}
\begin{aligned}
\frac{d}{dt} E_1(t)
&\le - \frac{K\mathcal{C}  \cos \bar{\alpha} \sin \beta}{2N\beta}   \left(\frac{8N\beta\left(D_\Omega + 2K \psi_u\sin\bar{\alpha} \right)}{K\mathcal{C}  \cos \bar{\alpha} \sin \beta} - \frac{4N\beta\left(D_\Omega + 2K \psi_u\sin\bar{\alpha} \right)}{K\mathcal{C}  \cos \bar{\alpha} \sin \beta} \right)\\
&= - 2(D_\Omega +  2K \psi_u \sin \bar{\alpha}) < 0.
\end{aligned}
\end{equation*}
This means that
\begin{equation}\label{E-7}
E_1(t)\le   \frac{8N\beta\left(D_\Omega + 2K \psi_u\sin\bar{\alpha} \right)}{K\mathcal{C}  \cos \bar{\alpha} \sin \beta}, \quad \text{for} \ t \ge t_*,
\end{equation}
where $t_*$ can be estimated as follows
\begin{equation*}
t_* = \frac{E_1(0) - \frac{8N\beta\left(D_\Omega + 2K \psi_u\sin\bar{\alpha} \right)}{K\mathcal{C}  \cos \bar{\alpha} \sin \beta}}{2(D_\Omega +  2K \psi_u \sin \bar{\alpha})} < \frac{\beta}{2(D_\Omega +  2K \psi_u \sin \bar{\alpha})}.
\end{equation*}

\noindent $\diamond$
For the other case that 
\begin{equation*}
E_1(0) \le \frac{8N\beta\left(D_\Omega + 2K \psi_u\sin\bar{\alpha} \right)}{K\mathcal{C}  \cos \bar{\alpha} \sin \beta},
\end{equation*}
it yields from \eqref{E-6} that
\begin{equation}\label{E-8}
E_1(t) \le \frac{8N\beta\left(D_\Omega + 2K \psi_u\sin\bar{\alpha} \right)}{K\mathcal{C}  \cos \bar{\alpha} \sin \beta}, \quad t\ge 0,
\end{equation}
Hence, we set $t_* = 0$ in this case.

Thus, we collect the results \eqref{E-7} and \eqref{E-8} in above two cases to conclude that there exists time $t_* \ge 0$ such that
\begin{equation*}
E_1(t) \le \frac{8N\beta\left(D_\Omega + 2K \psi_u\sin\bar{\alpha} \right)}{K\mathcal{C}  \cos \bar{\alpha} \sin \beta} < D^\infty, \quad t \ge t^*,
\end{equation*}
where we used the assumption \eqref{K_con} yielding
\begin{equation*}
K > \frac{8N\beta\left(D_\Omega + 2K \psi_u\sin\bar{\alpha} \right)}{D^\infty \mathcal{C}  \cos \bar{\alpha} \sin \beta} \quad \Longrightarrow \quad \frac{8N\beta\left(D_\Omega + 2K \psi_u\sin\bar{\alpha} \right)}{K\mathcal{C}  \cos \bar{\alpha} \sin \beta} < D^\infty.
\end{equation*}
This further implies
\begin{equation*}
D_\theta(t) < D^\infty, \quad t \ge t_*.
\end{equation*}
Moreover, one has
\begin{equation*}
D_\omega(t) \le \frac{32N^2\psi_u\beta^2(D_\Omega +  2K \psi_u \sin \bar{\alpha})}{\gamma K\mathcal{C}^2  \cos^2 \bar{\alpha} \sin^2 \beta}, \quad t \ge t_*.
\end{equation*}
\end{proof}

\subsection{Exponential frequency synchronization}\label{sec:3.2}
In this part, we will show the frequency diameter decays to zero exponentially fast, which results in the emergence of complete frequency synchronization. To this end, we consider system \eqref{inertia_fre} starting from $t_*$ mentioned in Lemma \ref{phs_small_bound}:
\begin{equation}\label{inertia_fre2}
\gamma \ddot{\omega}_i(t) + \dot{\omega}_i(t) = \frac{K}{N} \sum_{k=1}^N \frac{\psi_{ik}}{d_i} \cos (\theta_k(t) - \theta_i(t) + \alpha_{ik})(\omega_k(t) - \omega_i(t)), \quad t \ge t_*,  \ i=1,2,\ldots,N.
\end{equation}
From $a_i(t) = \dot{\omega}_i(t)$, one has
\begin{equation}\label{first_acce2}
\gamma \dot{a}_i(t) + a_i(t) = \frac{K}{N} \sum_{k=1}^N \frac{\psi_{ik}}{d_i} \cos (\theta_k(t) - \theta_i(t) + \alpha_{ik})(\omega_k(t) - \omega_i(t)), \quad t \ge t_*,  \ i=1,2,\ldots,N.
\end{equation}
Then, we define another energy function
\begin{equation}\label{energy_2}
E_2(t) := D_\omega(t) + \frac{\mathcal{C}\cos (D^\infty + \bar{\alpha})}{4N \psi_u}\gamma D_a(t) + 2\gamma^2 D_b(t),\quad t \ge t_*,
\end{equation}
which can be utilized to validate the exponential convergence of frequency diameter to zero.

Similarly, owing to the analyticity of solution, we divide the time interval $[t_*, +\infty)$ into a union of infinitely countable intervals:
\begin{equation}\label{com_time_divide2}
[t_*, +\infty) = \bigcup_{l=1}^{+\infty} T_l, \quad T_l = [s_{l-1}, s_l), \quad s_0 =t_*,
\end{equation}
such that the orders of oscillators' frequencies, accelerations and derivatives of accelerations are unchanged on each time interval $T_l$.

We first present a dissipative estimate on the dynamics of frequency diameter.
\begin{lemma}\label{second_w_eq}
Let $(\theta(t),\omega(t))$ be a solution to system \eqref{second_theta_system} with the network structure \eqref{com_net}, and suppose the Assumption $(\mathcal{A})$ holds. Then, we have
\begin{equation}\label{fredia_s_dynamics}
\begin{aligned}
\gamma \ddot{D}_\omega(t) + \dot{D}_\omega(t) \le - \frac{K \mathcal{C}\cos (D^\infty + \bar{\alpha})}{N}  D_\omega(t), \quad t \in T_l, \ l=1,2,\ldots,
\end{aligned}
\end{equation}
where $T_l$ is defined in \eqref{com_time_divide2}.
\end{lemma}
\begin{proof}
The proof is very similar to Lemma \ref{second_theta_eq}, so we omit the details. 
\end{proof}

Similar to Lemma \ref{second_theta_eq}, due to the singularities of second-order derivative of $D_\omega(t)$ in the whole time line, the differential inequality \eqref{fredia_s_dynamics} is not sufficient to show the decay of frequency diameter. Hence, we subsequently take into consideration the dynamics of jerk (the derivative of acceleration) diameter, which may control the term $\ddot{D}_\omega(t)$.
For this, we directly differentiate \eqref{first_acce2} with respect to time and get
\begin{equation}\label{inertia_acce}
\begin{aligned}
\gamma\ddot{a}_i(t) + \dot{a}_i(t) 
&= \frac{K}{N} \sum_{j=1}^N \frac{\psi_{ij}}{d_i}\cos(\theta_j(t) - \theta_i(t)+ \alpha_{ij}) (a_j(t) - a_i(t))\\
&- \frac{K}{N} \sum_{k=1}^N \frac{\psi_{ij}}{d_i} \sin (\theta_j(t) - \theta_i(t) + \alpha_{ij})(\omega_j(t) - \omega_i(t))^2, \quad t \ge t_*.
\end{aligned}
\end{equation}
Recalling $b_i(t) = \dot{a}_i(t)$, we transform \eqref{inertia_acce} into the following form
\begin{equation}\label{first_b}
\begin{aligned}
\gamma\dot{b}_i(t) + b_i(t)  
&= \frac{K}{N} \sum_{j=1}^N \frac{\psi_{ij}}{d_i}\cos(\theta_j(t) - \theta_i(t)+ \alpha_{ij}) (a_j(t) - a_i(t))\\
&- \frac{K}{N} \sum_{k=1}^N \frac{\psi_{ij}}{d_i} \sin (\theta_j(t) - \theta_i(t) + \alpha_{ij})(\omega_j(t) - \omega_i(t))^2, \quad t \ge t_*.
\end{aligned}
\end{equation}
\begin{lemma}\label{first_b_eq}
Let $(\theta(t),\omega(t))$ be a solution to system \eqref{second_theta_system} with the network structure \eqref{com_net}, and suppose the Assumption $(\mathcal{A})$ holds. Then, we have
\begin{equation}\label{bdia_f_dynamics2}
\begin{aligned}
\gamma\dot{D}_b(t) + D_b(t) \le  2K \psi_u D_a(t) +  \frac{\mu}{\gamma} D_\omega(t), \quad t \in T_l, \ l=1,2,\ldots,
\end{aligned}
\end{equation}
where $T_l$ is defined in \eqref{com_time_divide2} and $\mu$ is given in \eqref{mu}.
\end{lemma}
\begin{proof}
For any fixed time interval $T_l$ in \eqref{com_time_divide2} and any $t \in T_l$ with $l=1,2,\ldots$, there exists some indices $m$ and $n$ such that
\begin{equation}\label{F-3}
D_b(t) = b_m(t) - b_n(t).
\end{equation}
It yields from \eqref{first_b} that
\begin{equation}\label{F-4}
\begin{aligned}
&\gamma(\dot{b}_m(t) - \dot{b}_n(t)) + b_m(t) - b_n(t) \\
&= \frac{K}{N} \sum_{j=1}^N \frac{\psi_{mj}}{d_m}\cos(\theta_j - \theta_m+ \alpha_{mj}) (a_j - a_m) - \frac{K}{N} \sum_{j=1}^N \frac{\psi_{nj}}{d_n}\cos(\theta_j - \theta_n+ \alpha_{nj}) (a_j - a_n)\\
&\quad- \frac{K}{N} \sum_{j=1}^N \frac{\psi_{mj}}{d_m} [\sin (\theta_j - \theta_m) \cos \alpha_{mj} +\cos (\theta_j - \theta_m) \sin \alpha_{mj} ](\omega_j - \omega_m)^2\\
&\quad+ \frac{K}{N} \sum_{j=1}^N \frac{\psi_{nj}}{d_n} [\sin (\theta_j - \theta_n) \cos \alpha_{nj} +\cos (\theta_j - \theta_n) \sin \alpha_{nj}  ](\omega_j - \omega_n)^2.
\end{aligned}
\end{equation}
According to Lemma \ref{phs_small_bound}, we see from \eqref{F-4} that
\begin{equation*}
\begin{aligned}
\gamma\dot{D}_b(t) + D_b(t)&\le 2K \psi_u D_a(t) + 2K \psi_u [D_\theta(t)  + \sin \bar{\alpha}] D^2_\omega(t)\\
&\le 2K \psi_u D_a(t) + 2K \psi_u [D^\infty  + \sin \bar{\alpha}] \frac{32N^2\psi_u\beta^2(D_\Omega +  2K \psi_u \sin \bar{\alpha})}{\gamma K\mathcal{C}^2  \cos^2 \bar{\alpha} \sin^2 \beta}D_\omega(t), \\
&=2K \psi_u D_a(t) +  \frac{\mu}{\gamma} D_\omega(t),\quad t \in T_l,
\end{aligned}
\end{equation*}
where the relation $|\sin x| \le |x|$ is exploited and constant $\mu$ is defined in \eqref{mu}. 
\end{proof}

Finally, we will show the exponential convergence to frequency synchronization.
\begin{lemma}\label{fredia_decay}
Let $(\theta(t),\omega(t))$ be a solution to system \eqref{second_theta_system} with the network structure \eqref{com_net}, and suppose the Assumption $(\mathcal{A})$ holds. Then, we have
\begin{equation*}
\frac{d}{dt}E_2(t) \le -\frac{K \mathcal{C}\cos (D^\infty + \bar{\alpha})}{4N}  E_2(t), \quad \text{a.e.} \ t \ge t_*.
\end{equation*}
Moreover, we have
\begin{equation*}
D_\omega(t) \le E_2(t_*) e^{-\frac{K \mathcal{C}\cos (D^\infty + \bar{\alpha})}{4N}(t-t_*)}, \quad t \ge t_*.
\end{equation*}
\end{lemma}
\begin{proof}
Recall the results in \eqref{fredia_s_dynamics} and \eqref{bdia_f_dynamics2}, and applying the similar arguments as in \eqref{accedia_f_dynamics}, we have for any $t \in T_l$ in \eqref{com_time_divide2} with $l=1,2,\ldots$, 
\begin{align}
&\gamma \ddot{D}_\omega(t) + \dot{D}_\omega(t) \le - \frac{K \mathcal{C}\cos (D^\infty + \bar{\alpha})}{N}  D_\omega(t), \label{F-5a}\\
&\gamma\dot{D}_a(t) + D_a(t) \le 2K \psi_uD_\omega(t), \label{F-5b}\\
&\gamma\dot{D}_b(t) + D_b(t) \le  2K \psi_u D_a(t) +  \frac{\mu}{\gamma} D_\omega(t). \label{F-5c}
\end{align}
Then, we combine \eqref{F-5a}, \eqref{F-5b} and \eqref{F-5c} together to have
\begin{equation*}
\begin{aligned}
&\gamma \ddot{D}_\omega(t) + \dot{D}_\omega(t) + \frac{\mathcal{C}\cos (D^\infty + \bar{\alpha})}{4N \psi_u}\gamma\dot{D}_a(t) + \frac{\mathcal{C}\cos (D^\infty + \bar{\alpha})}{4N \psi_u}D_a(t)+ 2\gamma^2\dot{D}_b(t) + 2\gamma D_b(t) \\
&\le - \frac{K \mathcal{C}\cos (D^\infty + \bar{\alpha})}{N}  D_\omega(t) + \frac{K \mathcal{C}\cos (D^\infty + \bar{\alpha})}{2N}D_\omega(t) + 4\gamma K \psi_u D_a(t) + 2 \mu D_\omega(t),  \quad t \in T_l. 
\end{aligned}
\end{equation*}
This also yields
\begin{equation*}
\begin{aligned}
&\frac{d}{dt} \left( D_\omega(t) + \frac{\mathcal{C}\cos (D^\infty + \bar{\alpha})}{4N \psi_u}\gamma D_a(t) + 2\gamma^2 D_b(t)\right)\\
&\le -\gamma \ddot{D}_\omega(t) -\frac{\mathcal{C}\cos (D^\infty + \bar{\alpha})}{4N \psi_u}D_a(t) -  2\gamma D_b(t)\\
&\quad - \frac{K \mathcal{C}\cos (D^\infty + \bar{\alpha})}{N}  D_\omega(t) + \frac{K \mathcal{C}\cos (D^\infty + \bar{\alpha})}{2N}D_\omega(t) + 4\gamma K \psi_u D_a(t) + 2 \mu D_\omega(t)\\
&\le - \left( \frac{K \mathcal{C}\cos (D^\infty + \bar{\alpha})}{2N} - 2 \mu \right)D_\omega(t) - \left(\frac{\mathcal{C}\cos (D^\infty + \bar{\alpha})}{4N \psi_u} -  4\gamma K \psi_u\right) D_a(t) - \gamma D_b(t) \\
&\le -\frac{K \mathcal{C}\cos (D^\infty + \bar{\alpha})}{4N} D_\omega(t) - \frac{\mathcal{C}\cos (D^\infty + \bar{\alpha})}{8N \psi_u}D_a(t) -\gamma D_b(t).
\end{aligned}
\end{equation*}
Here, we used the following relations
\begin{equation*}
\begin{aligned}
&|\ddot{D}_\omega(t)| \le D_b(t), \quad
 K > \frac{8N\mu}{\mathcal{C} \cos (D^\infty + \bar{\alpha})}\quad  \Longrightarrow \quad \frac{K \mathcal{C}\cos (D^\infty + \bar{\alpha})}{4N} - 2\mu >0 , \\
&\gamma K < \frac{\mathcal{C}\cos (D^\infty + \bar{\alpha})}{32N \psi^2_u} \quad \Longrightarrow \quad \frac{\mathcal{C}\cos (D^\infty + \bar{\alpha})}{8N \psi_u} - 4\gamma K \psi_u >0.\\
\end{aligned}
\end{equation*}
Moreover, under the assumption that
\begin{align*}
& \gamma K <\frac{2N }{\mathcal{C}\cos (D^\infty + \bar{\alpha})} \quad \Longrightarrow \quad \frac{1}{2 K\psi_u} > \frac{\mathcal{C}\cos (D^\infty + \bar{\alpha})}{4N \psi_u}\gamma \quad \text{and} \quad \frac{4N}{K \mathcal{C}\cos (D^\infty + \bar{\alpha})} > 2 \gamma, \\
\end{align*}
we obtain that
\begin{equation*}
\begin{aligned}
&\frac{d}{dt} \left( D_\omega(t) + \frac{\mathcal{C}\cos (D^\infty + \bar{\alpha})}{4N \psi_u}\gamma D_a(t) + 2\gamma^2 D_b(t)\right)\\
&\le-\frac{K \mathcal{C}\cos (D^\infty + \bar{\alpha})}{4N} \left(D_\omega(t) +\frac{1}{2 K\psi_u}D_a(t) + \frac{4N}{K \mathcal{C}\cos (D^\infty + \bar{\alpha})}\gamma D_b(t)\right)\\
&\le-\frac{K \mathcal{C}\cos (D^\infty + \bar{\alpha})}{4N} \left( D_\omega(t) + \frac{\mathcal{C}\cos (D^\infty + \bar{\alpha})}{4N \psi_u}\gamma D_a(t) + 2\gamma^2 D_b(t)\right), \quad t \in T_l.
\end{aligned}
\end{equation*}
Therefore from \eqref{energy_2} and the arbitrary choice of $T_l$, we conclude that
\begin{equation*}
\frac{d}{dt}E_2(t) \le -\frac{K \mathcal{C}\cos (D^\infty + \bar{\alpha})}{4N}  E_2(t), \quad \text{a.e.} \ t \ge t_*.
\end{equation*}
Furthermore, we have
\begin{equation*}
\begin{aligned}
D_\omega(t) \le E_2(t) \le E_2(t_*) e^{-\frac{K \mathcal{C}\cos (D^\infty + \bar{\alpha})}{4N}(t-t_*)}, \quad t \ge t_*.
\end{aligned}
\end{equation*}
\end{proof}

Now, we are ready to provide the proof of Theorem \ref{com_main}.\newline

\noindent {\bf Proof of Theorem \ref{com_main}:} The proof of Theorem \ref{com_main} immediately follows from Lemma \ref{phs_small_bound} and Lemma \ref{fredia_decay}.

\section{Numerical Simulations}\label{sec:5}
\setcounter{equation}{0}
In this section, we present some simulations to sustain the two main results in Theorem \ref{com_main}. We take four oscillators ($N=4$) as an example and use the classical fourth-order Runge-Kutta numerical method.

For the system \eqref{second_theta_system} with the network structure \eqref{com_net}, we consider a graph registered by
\begin{equation}\label{G-1}
\Psi = (\psi_{ik}) = 
\begin{pmatrix}
1 & 1 & 0 & 0\\
1 & 1 & 0 & 0\\
0 & 1 & 1 & 0\\
0 & 1 & 1 & 1
\end{pmatrix}
,
\end{equation}
whose diameter equals to two.
The non-homogeneous dampings $d_i$'s are randomly chosen in $(0.9, 1.0)$ and natural frequencies $\Omega_i$'s in $(-0.005,0.005)$:
\begin{equation*}
\begin{aligned}
&d = (d_1, d_2, d_3, d_4) = (0.9775, 0.9165, 0.9912, 0.9319),\\
&\Omega = (0.0013,  -0.0040, -0.0022, 0.0005).
\end{aligned}
\end{equation*}
By calculations, we have 
\begin{equation*}
\mathcal{C} \approx 1.0089 > 0, \quad D_\Omega \approx 0.0057, \quad \psi_u \approx 1.0911.
\end{equation*}
The initial phases and frequencies are randomly selected in $(0, \frac{2}{3} \pi)$ and $(-0.1, 0.1)$, respectively:
\begin{equation*}
\theta(0) = (2.0742, 0.0706, 0.8886, 1.0262), \quad \omega(0) = (0.0701, 0.0117, 0.0804, -0.0161),
\end{equation*}
which yield $D_\theta(0) = 2.0036$ and $D_\omega(0) = 0.0965$.
Moreover, we set 
\begin{equation*}
K = 780, \quad \gamma = 10^{-6}, \quad \alpha_{ik} = 10^{-6} \ \text{for any $i \ne k$}, \quad \beta = \frac{5}{6}\pi, \quad D^\infty = 0.1. 
\end{equation*} 
This means $\bar{\alpha} = 10^{-6}$.
Based on above settings, the Assumption $(\mathcal{A})$ is fulfilled. 
Then, frequency synchronization can be observed from Figure \ref{Fig1:phs1} and \ref{Fig1:fre1}, which supports Theorem \ref{com_main}. For the ease of reading,, we present the detailed formation in time interval $[0.125,0.15]$ displayed in Figure \ref{Fig1:phs1} and \ref{Fig1:fre1} by Figure \ref{Fig1:local_phs1} and \ref{Fig1:local_fre1}.

\begin{figure}[H]
\centering
\mbox{
\subfigure[]{\includegraphics[width=0.48\textwidth]{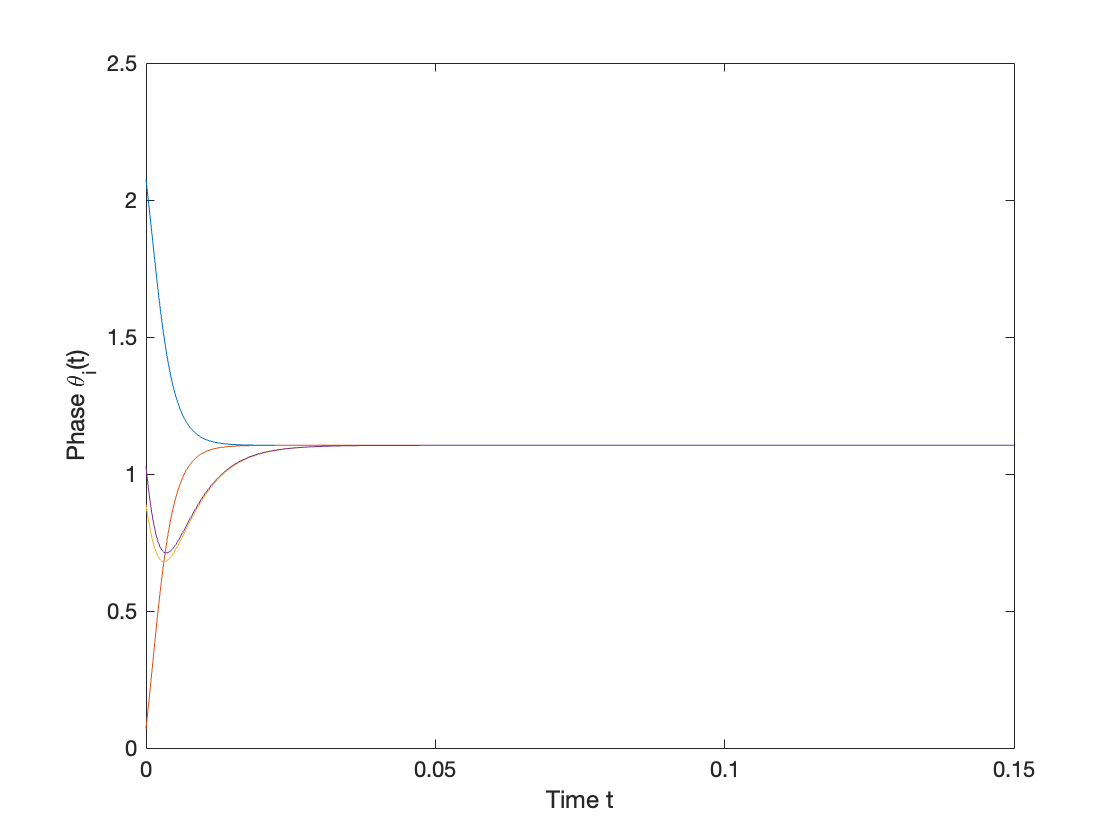}\label{Fig1:phs1}}

\subfigure[]{\includegraphics[width=0.48\textwidth]{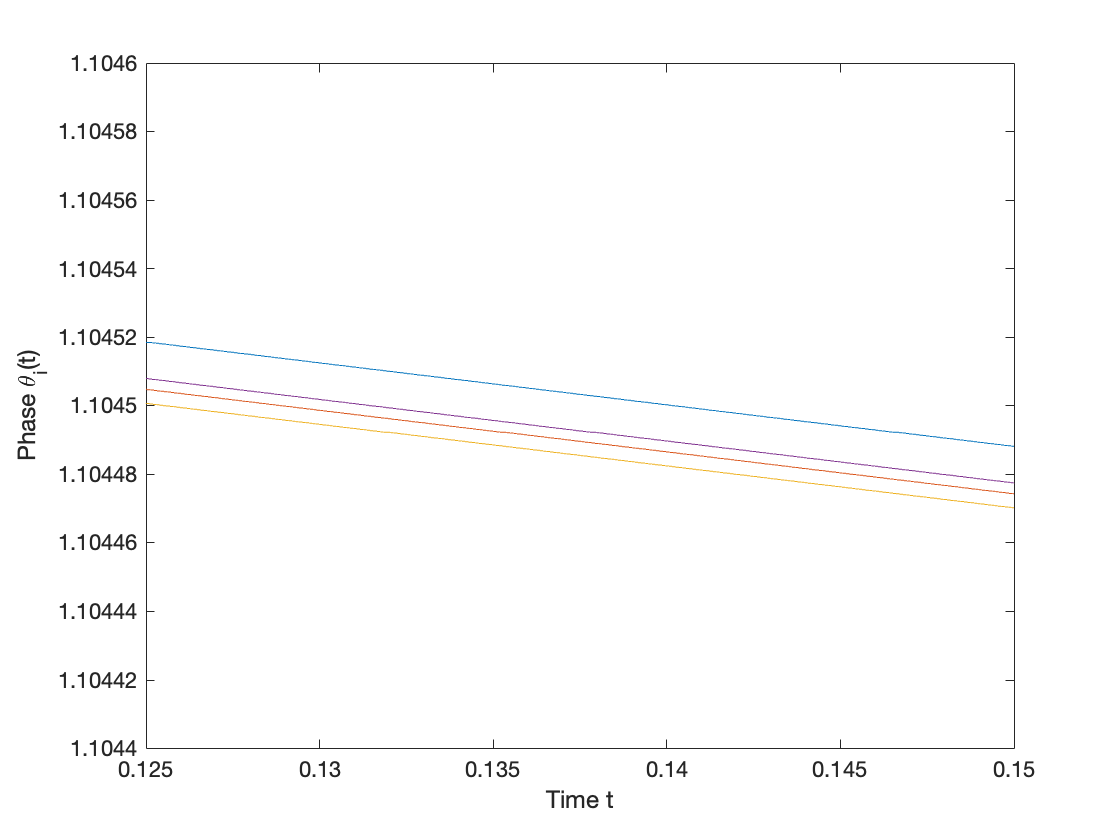}\label{Fig1:local_phs1}}
}\\

\mbox{
\subfigure[]{\includegraphics[width=0.48\textwidth]{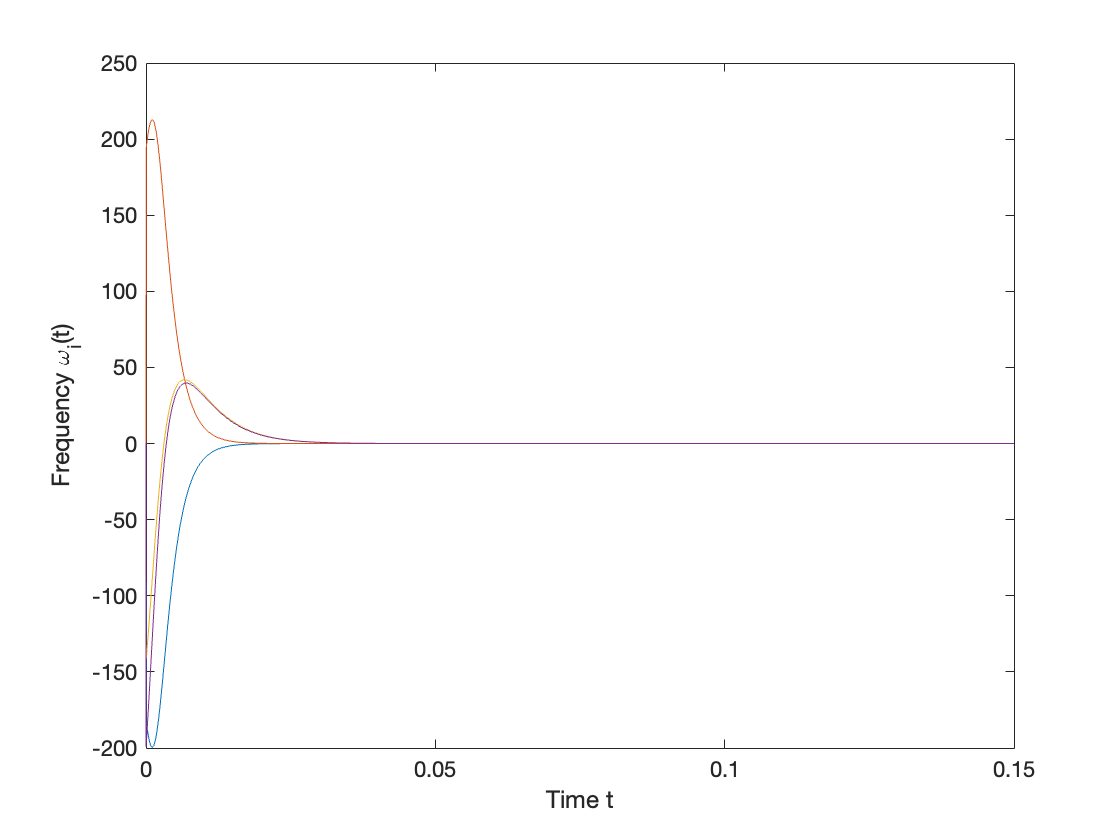}\label{Fig1:fre1}}

\subfigure[]{\includegraphics[width=0.48\textwidth]{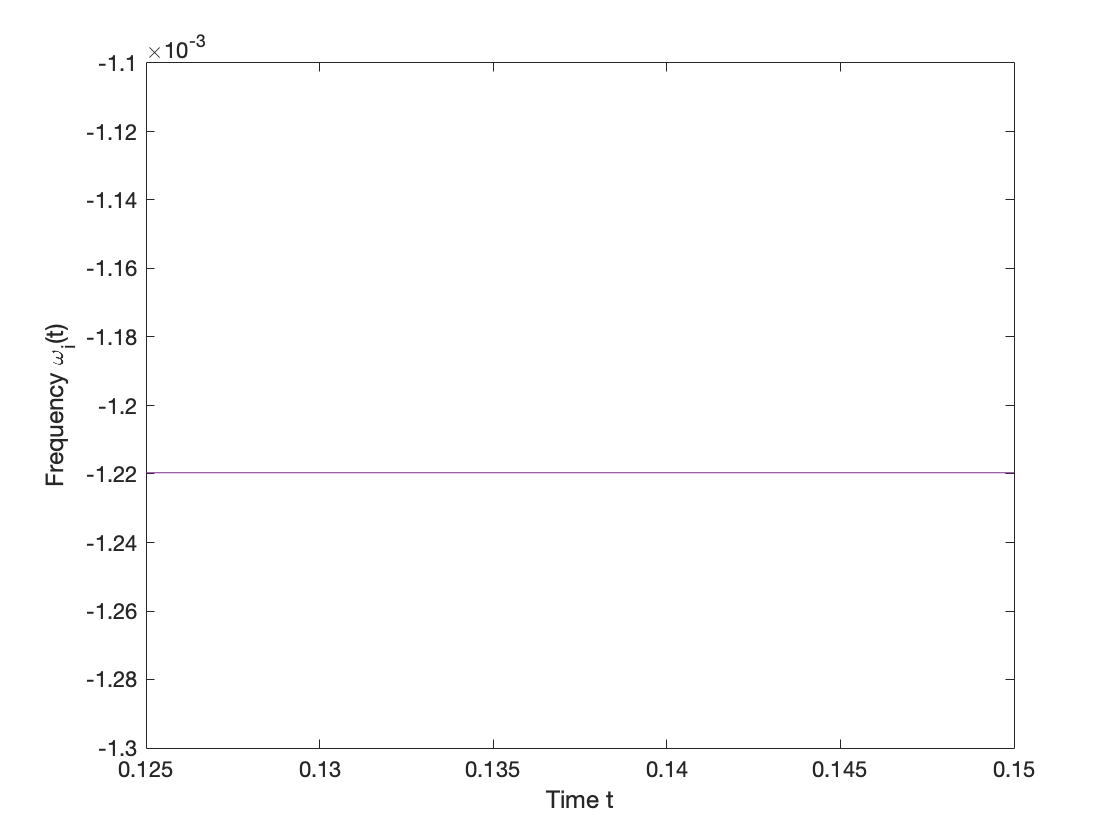}\label{Fig1:local_fre1}}
}

\caption{Complete synchronization on the network \eqref{G-1}}
\label{Fig1}
\end{figure}

\section{Summary}\label{sec:6}
In this paper, we study the synchronized behavior of the Kuramoto model with inertia and frustration effects on asymmetric networks. In particular, we consider the network with depth not greater than two, and present sufficient conditions leading to the exponential convergence of complete frequency synchronization. Our sufficient frameworks can be guaranteed by a regime in terms of small effects of inertia and frustration, and large coupling. The approach we used is mainly based on the derived first-order Gronwall type inequalities of diameter functions. We believe this method is able to be applied to many other second-order models with more general connected networks, and this will be discussed in our future work.

\section*{Acknowledgments}

The work of T. Zhu is supported by the National Natural Science Foundation of China (Grant No. 12201172), the Natural Science Research Project of Universities in Anhui Province, China (Project Number: 2022AH051790) and  the Talent Research Fund of Hefei University, China (Grant/Award Number: 21-22RC23). The work of X. Zhang is supported by the National Natural Science Foundation of China (Grant No. 12471213).

\end{document}